\documentclass[11pt]{article}

\usepackage{amsmath,amssymb,amsthm}
\usepackage[margin=1.05in]{geometry}
\usepackage[colorlinks=true,linkcolor=blue,citecolor=blue,urlcolor=blue]{hyperref}
\usepackage{enumitem}
\usepackage{mathtools}

\newtheorem{theorem}{Theorem}[section]
\newtheorem{lemma}[theorem]{Lemma}
\newtheorem{proposition}[theorem]{Proposition}
\newtheorem{corollary}[theorem]{Corollary}
\newtheorem{remark}[theorem]{Remark}
\newtheorem{definition}[theorem]{Definition}

\newcommand{\CC}{\mathbb C}
\newcommand{\Hess}{\operatorname{Hess}}
\newcommand{\adj}{\operatorname{adj}}
\newcommand{\grad}{\nabla}
\newcommand{\tr}{\operatorname{tr}}

\title{The Quartic Hessian Conjecture in Dimension Four}
\author{Zixiang Ni\\
\href{mailto:akie.camellia@gmail.com}{\texttt{akie.camellia@gmail.com}}}
\date{\today}

\begin{document}
\maketitle

\begin{abstract}
The Hessian conjecture asks whether a polynomial with nonzero constant Hessian determinant has a polynomial gradient inverse.  It is known in dimensions at most three, false in dimensions at least five, and open in dimension four. We prove its four-variable quartic case.

The top homogeneous part has zero Hessian determinant and, by the four-dimensional homogeneous Hesse theorem, is a cone.  We divide its cone representative into three exhaustive types: a genuinely ternary quartic with nonzero ternary Hessian, a genuinely binary quartic, and a fourth power of a linear form.  In the first type, the degree-seven determinant equation forces the cubic part to be affine-linear in the cone direction.  In the binary type, the degree-six equation gives a constant null direction in a two-variable Hessian of the cubic part.  In the unary type, the degree-five equation and a constant-direction lemma give the same conclusion.  Every type therefore reduces to
\[ f=P(x_1,x_2,x_3)+x_4Q(x_1,x_2,x_3)+a x_4^2, \qquad \deg Q\leq2. \]
We prove, independently of the degree or top part of \(P\), that every constant-Hessian polynomial of this form has a polynomial gradient inverse. The branch \(a\ne0\) descends from the known three-dimensional Hessian conjecture after a Schur complement.  When \(a=0\), an isotropic-cone rank analysis eliminates rank two, solves the rank-one exception by an explicit triangular inverse, and reduces rank zero to the two-dimensional Hessian conjecture.  The coupled degree-six identity is retained throughout; no component with respect to a fixed quadratic form is separated.
\end{abstract}

\tableofcontents

\section{Introduction}

For \(f\in\CC[x_1,\ldots,x_n]\), write
\[
 \Hess(f)=(\partial_i\partial_jf)_{i,j},
 \qquad \grad f=(\partial_1f,\ldots,\partial_nf).
\]

\begin{definition}[Hessian conjecture, \(\mathrm{HC}_n\)]
If \(\det\Hess(f)\in\CC^\times\), then
\(\grad f:\CC^n\to\CC^n\) is a polynomial automorphism.
\end{definition}

The conjecture is elementary for \(n=1\), was proved for \(n=2\) by Dillen
\cite{Dillen1991}, and was proved for \(n=3\) by de Bondt
\cite{deBondt2015}.  It is false for \(n\geq5\) \cite{MengYang2026}, and is
open only for \(n=4\).  Moreover,
\(\mathrm{HC}_4\) implies the plane Jacobian conjecture by the standard
doubling construction \cite{Meng2006}.  In degree at most two the gradient is
affine, while a cubic potential has quadratic gradient and is covered by
Wang's theorem on quadratic Keller maps \cite{Wang1980}.  Thus degree four is
the first unresolved degree in dimension four.  A related boundary case was
already known: Hubbers proved the Jacobian conjecture for cubic-homogeneous
Keller maps in dimension four \cite{Hubbers1994}.  After normalizing the
invertible linear part of the gradient, this covers potentials of the form
\(f_4+f_2+f_1+f_0\), with \(f_i\) homogeneous of degree \(i\): if
\(A=\Hess(f_2)\), then constancy of the Hessian determinant gives
\(\det A\ne0\), and
\(A^{-1}(\grad f-\grad f(0))=x+H_3(x)\) with \(H_3\) cubic homogeneous.  The
result below allows an arbitrary cubic part as well, and hence treats every
quartic potential.

\begin{theorem}[Main theorem]\label{thm:main}
Let \(f\in\CC[x_1,x_2,x_3,x_4]\) have degree four.  If
\[
 \det\Hess(f)\in\CC^\times,
\]
then \(\grad f\) is a polynomial automorphism of \(\CC^4\).
\end{theorem}

The proof has two logically separate parts.  First, the top homogeneous form is
classified and each possible type is shown to produce a constant direction in
the cubic part.  Second, a uniform proposition in that direction proves
invertibility without any nondegeneracy assumption on the top form.

\begin{remark}[The corrected degree-six equation]\label{rem:degree6-correction}
In the genuinely ternary branch, write
\[
 f_4=\varphi(x_1,x_2,x_3),\qquad
 f_3=g(x_1,x_2,x_3)+x_4h(x_1,x_2,x_3),
\]
and \(A=\Hess_3(\varphi)\).  If the \(x_4^2\)-coefficient of \(f_2\) is \(a\),
the degree-six determinant equation is
\[
 2a\det A-(\grad h)^{\mathsf T}\adj(A)\grad h=0.
\]
The two terms may cancel: \(h\) is a fixed polynomial, not a free parameter.
The proof never separates them and does not attempt to deduce \(a=0\).
\end{remark}

\section{Preliminaries}\label{sec:prelim}

\subsection{Coordinate changes and cone tops}

All changes of variables below are constant invertible linear maps or source
translations.  If \(\widetilde f=f\circ T\), then
\[
 \grad\widetilde f(x)=T^{\mathsf T}(\grad f)(Tx),\qquad
 \Hess(\widetilde f)(x)=T^{\mathsf T}\Hess(f)(Tx)T.
\]
Thus the Hessian determinant changes only by a nonzero constant factor and
gradient invertibility is preserved.  Source translations have the analogous
affine-conjugacy property and do not alter the top homogeneous part.

\begin{lemma}[Top homogeneous Hessian]\label{lem:top-hessian}
Let \(f=\sum_{k=0}^d f_k\) be the decomposition into homogeneous parts, with
\(d\geq3\).  If \(\det\Hess(f)\) is constant, then
\[
 \det\Hess(f_d)\equiv0.
\]
\end{lemma}

\begin{proof}
The entries of \(\Hess(f_k)\) have degree \(k-2\).  The homogeneous component
of maximal possible degree \(n(d-2)\) in the determinant is obtained by taking
\(\Hess(f_d)\) in every entry, hence equals \(\det\Hess(f_d)\).  A constant
determinant has no positive-degree component.
\end{proof}

\begin{theorem}[Homogeneous Hesse theorem in low dimension]\label{thm:hesse}
Let \(h\) be homogeneous in \(2\leq n\leq4\) variables over \(\CC\).  If
\(\det\Hess(h)=0\), then after a constant invertible linear change, \(h\)
depends on at most \(n-1\) variables.
\end{theorem}

\begin{proof}
This is the low-dimensional case of the classical theorem of Gordan and
Noether \cite{GordanNoether1876}; see de Bondt
\cite[Theorem~3.1]{deBondtSingular2018} for a modern statement and proof in
dimensions at most four.
\end{proof}

\begin{lemma}[Binary forms with zero Hessian]\label{lem:binary-zero}
Let \(h\in\CC[X,Y]\) be homogeneous of degree \(d\geq2\).  Then
\[
 h_{XX}h_{YY}-h_{XY}^2\equiv0
 \quad\Longleftrightarrow\quad
 h=\kappa L^d
\]
for some \(\kappa\in\CC\) and linear form \(L\), with \(h=0\) allowed.
\end{lemma}

\begin{proof}
On \(Y\ne0\), write \(h(X,Y)=Y^dp(t)\), \(t=X/Y\).  Direct differentiation
gives
\[
 \det\Hess_2(h)
 =(d-1)Y^{2d-4}\bigl(dpp''-(d-1)(p')^2\bigr).
\]
Suppose \(p\ne0\) and put \(q=p'/p\in\CC(t)\).  Dividing the vanishing
equation by \(p^2\) gives
\[
 d q'+q^2=0.
\]
If \(q=0\), then \(p\) is constant.  Otherwise
\((1/q)'=1/d\), so \(q=d/(t-\alpha)\) for some \(\alpha\in\CC\).  Therefore
\[
 \left(\frac{p}{(t-\alpha)^d}\right)'=0,
\]
and the constant field of \(\CC(t)\) is \(\CC\).  Thus
\(p=\kappa(t-\alpha)^d\).  Homogenization gives the assertion.  The case
\(p=0\) is immediate, as is the converse.
\end{proof}

\subsection{Matrix identities and inverse descent}

\begin{lemma}[Block determinant]\label{lem:block-det}
For a square matrix \(B\) and a column vector \(v\) of the same size,
\[
 \det\begin{pmatrix}B&v\\v^{\mathsf T}&0\end{pmatrix}
 =-v^{\mathsf T}\adj(B)v.
\]
\end{lemma}

\begin{proof}
Expand along the last row and column, or use the polynomial continuation of
the Schur-complement formula from invertible \(B\).
\end{proof}

\begin{lemma}[Adjugate quadratic form]\label{lem:adj-quadratic}
For a symmetric \(3\times3\) matrix \(M\) and column vectors \(x,\ell\),
\[
 (Mx+\ell)^{\mathsf T}\adj(M)(Mx+\ell)
 =\det(M)\bigl(x^{\mathsf T}Mx+2\ell^{\mathsf T}x\bigr)
  +\ell^{\mathsf T}\adj(M)\ell.
\]
\end{lemma}

\begin{proof}
Expand by bilinearity and use
\(M\adj(M)=\adj(M)M=\det(M)I\), together with symmetry.
\end{proof}

\begin{lemma}[Descent of a polynomial inverse]\label{lem:descent}
Let \(R=\CC[t_1,\ldots,t_m]\), \(K=\operatorname{Frac}(R)\), and
\(F\in R[x_1,\ldots,x_n]^n\).  Assume
\[
 \det J_xF\in R^\times=\CC^\times,\qquad
 F^{-1}\in K[y_1,\ldots,y_n]^n.
\]
Then \(F^{-1}\in R[y_1,\ldots,y_n]^n\).
\end{lemma}

\begin{proof}
Subtract \(F(0)\) from the target and let \(L=J_xF(0)\).  Since
\(\det L\in\CC^\times\), an affine target change reduces to
\[
 \widetilde F(x)=x+H(x),\qquad
 H\in R[x]^n,\qquad \operatorname{ord}_xH\geq2.
\]
The unique formal inverse has the expansion
\[
 \widetilde F^{-1}(y)=y+\sum_{d\geq2}G_d(y),
\]
where \(G_d\) is homogeneous of \(y\)-degree \(d\).  Comparing total degree in
\(\widetilde F(\widetilde F^{-1}(y))=y\) recursively expresses each \(G_d\)
over \(R\).  The inverse over \(K\) is this same formal inverse and is a
polynomial, so only finitely many \(G_d\) are nonzero.  Undoing the affine
target changes proves the claim.
\end{proof}

\begin{lemma}[Transfer to characteristic-zero fields]
\label{lem:charzero-transfer}
If \(\mathrm{HC}_n\) holds over \(\CC\), then it holds over every field of
characteristic zero.
\end{lemma}

\begin{proof}
Let \(K\) have characteristic zero and let \(f\in K[x_1,\ldots,x_n]\) have
nonzero constant Hessian determinant.  Choose a subfield
\(k\subset K\), finitely generated over \(\mathbb Q\), which contains the
coefficients of \(f\) and the inverse of that determinant.  Every finitely
generated characteristic-zero field embeds in \(\CC\): choose a
transcendence basis over \(\mathbb Q\), send it to algebraically independent
complex numbers, and extend across the resulting finite algebraic extension.
After such an embedding, \(\mathrm{HC}_n\) over \(\CC\) supplies a polynomial
inverse of \(\grad f\).

To descend it, subtract \(\grad f(0)\) from the target and normalize the
invertible linear part.  The coefficients of the unique formal inverse are
then determined recursively from coefficients in \(k\), exactly as in Lemma
\ref{lem:descent}; hence they lie in \(k\).  The polynomial inverse over
\(\CC\) coincides with this formal inverse, so only finitely many of these
coefficients are nonzero.  It therefore belongs to \(k[y]^n\), and hence gives
an inverse over \(K\).
\end{proof}

\section{Classification of quartic cone tops}\label{sec:classification}

Let \(f=f_4+f_3+f_2+f_1+f_0\) have degree four and nonzero constant Hessian
determinant.  Lemma \ref{lem:top-hessian} and Theorem \ref{thm:hesse} allow a
linear change after which
\[
 f_4=\varphi(x_1,x_2,x_3).
\]
There are three exhaustive cases.

\begin{enumerate}[leftmargin=2em]
\item If \(\det\Hess_3(\varphi)\not\equiv0\), the top is
  \emph{genuinely ternary}.
\item If \(\det\Hess_3(\varphi)\equiv0\), Theorem \ref{thm:hesse} in three
  variables makes \(\varphi\) binary.  If its binary Hessian is nonzero, the
  top is \emph{genuinely binary}.
\item If the binary Hessian also vanishes, Lemma \ref{lem:binary-zero} gives
  \(\varphi=cL^4\).  Since \(f\) has degree four, \(c\ne0\); after a linear
  change the top is \emph{unary}, \(f_4=cx_1^4\).
\end{enumerate}

The next two sections show that all three cases lead to the same structural
form.

\section{The genuinely ternary top}\label{sec:ternary}

Assume
\[
 f_4=\varphi(x_1,x_2,x_3),\qquad
 A=\Hess_3(\varphi),\qquad \det A\not\equiv0.
\]
Put
\[
 H_4=\Hess(f_4)=\begin{pmatrix}A&0\\0&0\end{pmatrix},
 \qquad H_k=\Hess(f_k).
\]

\subsection{The degree-seven equation}

The degree-seven part of \(\det(H_4+H_3+H_2)\) contains three entries of
\(H_4\) and one of \(H_3\).  It is the first variation
\[
 \tr\bigl(\adj(H_4)H_3\bigr).
\]
The adjugate of \(H_4\) has only one possibly nonzero entry, namely
\((\adj H_4)_{44}=\det A\).  Therefore
\[
 0=\det A\,\partial_{44}f_3.
\]
The polynomial ring is a domain and \(\det A\ne0\), so
\[
 f_3=g(x_1,x_2,x_3)+x_4h(x_1,x_2,x_3),
\]
where \(g\) is cubic and \(h\) is quadratic.

\subsection{The coupled degree-six equation}

Write
\[
 f_2=q(x_1,x_2,x_3)+x_4\ell(x_1,x_2,x_3)+a x_4^2,
\]
where \(q\) is quadratic, \(\ell\) is linear, and \(a\in\CC\).  Degree six
has exactly two types of contribution:
\[
 (2,2,2,0)\quad\text{and}\quad(2,2,1,1),
\]
where the numbers denote entry degrees.  The first gives
\(2a\det A\).  In the second, the two \(H_4\)-entries exhaust the first three
coordinates except for one row and column, while the last row and column come
from \(\Hess(x_4h)\).  The block expansion gives
\[
 -(\grad h)^{\mathsf T}\adj(A)\grad h.
\]
Thus
\begin{equation}\label{eq:coupled-degree-six}
 2a\det A-(\grad h)^{\mathsf T}\adj(A)\grad h=0.
\end{equation}
No separation of its two fixed summands is made.

Combining the preceding forms and absorbing \(f_1,f_0\) gives the following.

\begin{proposition}[Ternary structural reduction]\label{prop:ternary-structure}
In the genuinely ternary case, after the cone-coordinate change,
\[
 f=P(x_1,x_2,x_3)+x_4Q(x_1,x_2,x_3)+a x_4^2,
 \qquad \deg Q\leq2.
\]
\end{proposition}

\section{The degenerate cone tops}\label{sec:degenerate}

\subsection{The genuinely binary top}

Assume
\[
 f_4=\varphi(x_1,x_2),\qquad
 A=\Hess_2(\varphi),\qquad \det A\not\equiv0.
\]
Put \(X=(x_1,x_2)\) and \(Z=(x_3,x_4)\).  Then
\[
 H_4=\begin{pmatrix}A&0\\0&0\end{pmatrix}.
\]

\begin{lemma}[Binary-top degree-six identity]\label{lem:binary-degree-six}
\[
 [\det\Hess(f)]_6
 =\det A\cdot\det\Hess_Z(f_3).
\]
\end{lemma}

\begin{proof}
A degree-six determinant term has entry degrees \(2+2+1+1\) or
\(2+2+2+0\).  A nonzero permutation term cannot contain three \(H_4\)-entries,
because \(H_4\) is supported on only the two \(X\)-rows and columns.  With two
\(H_4\)-entries, those entries exhaust both \(X\)-rows and columns, leaving
exactly the \(Z\times Z\) minor of \(H_3\).  Summing the two block minors gives
the displayed product with positive sign.
\end{proof}

Since the determinant is constant and the polynomial ring is a domain,
\[
 \det\Hess_Z(f_3)=0.
\]
Write
\[
 \Hess_Z(f_3)=\begin{pmatrix}\alpha&\beta\\\beta&\delta\end{pmatrix},
\]
where \(\alpha,\beta,\delta\) are linear forms and
\(\alpha\delta=\beta^2\).  If \(\alpha\ne0\), primality of the linear
polynomial \(\alpha\) gives \(\beta=\mu\alpha\) and
\(\delta=\mu^2\alpha\) for some \(\mu\in\CC\).  If \(\alpha=0\), then
\(\beta=0\).  Thus, including the zero matrix, there are a linear form
\(\lambda\) and a constant vector \(r\in\CC^2\) such that
\[
 \Hess_Z(f_3)=\lambda rr^{\mathsf T}.
\]
Choose a nonzero constant \(w\in\ker(r^{\mathsf T})\).  Then
\(\partial_w^2f_3=0\).  A constant change among \(x_3,x_4\) makes \(w\) the
fourth coordinate direction.  The top remains independent of that coordinate,
and \(f_3\) is affine-linear in it.  Since \(f_2\) is quadratic, the whole
polynomial has the form
\begin{equation}\label{eq:common-structure}
 f=P(x_1,x_2,x_3)+x_4Q(x_1,x_2,x_3)+a x_4^2,
 \qquad \deg Q\leq2.
\end{equation}

\subsection{The unary top}

Assume now
\[
 f_4=cx_1^4,\qquad c\ne0,\qquad y=(x_2,x_3,x_4).
\]
Here \(H_4=\operatorname{diag}(12cx_1^2,0,0,0)\).  Every nonzero degree-five
determinant term contains one entry of \(H_4\) and three entries of \(H_3\);
terms with two \(H_4\)-entries vanish.  Hence
\[
 [\det\Hess(f)]_5
 =12cx_1^2\det\Hess_y(f_3)=0,
 \qquad \det\Hess_y(f_3)=0.
\]

\begin{lemma}[Constant-direction lemma]\label{lem:constant-direction}
Let \(F(x_1;y_1,y_2,y_3)\) be homogeneous of degree three.  If
\[
 \det\Hess_y(F)=0,
\]
then there is a nonzero constant vector \(v\in\CC^3\) such that
\(\partial_v^2F=0\).
\end{lemma}

\begin{proof}
Write
\[
 F=g_3(y)+x_1q_2(y)+x_1^2\ell_1(y)+\gamma x_1^3,
 \qquad A(y)=\Hess(g_3),\qquad M=\Hess(q_2).
\]
Then \(\Hess_y(F)=A(y)+x_1M\).  At \(x_1=0\), \(\det A=0\).
Theorem \ref{thm:hesse} makes \(g_3\) depend on at most two \(y\)-variables
after a constant change.

Suppose first that the essential variable number of \(g_3\) is two.  Write
\[
 A(y)=\begin{pmatrix}A_2&0\\0&0\end{pmatrix},\qquad
 M=\begin{pmatrix}B&m\\m^{\mathsf T}&\mu\end{pmatrix}.
\]
Lemma \ref{lem:binary-zero} shows
\(\Delta=\det A_2\not\equiv0\); otherwise \(g_3\) would be a cube of one
linear form.  The exact expansion
\[
 \det\left(
 \begin{pmatrix}A_2&0\\0&0\end{pmatrix}+tM\right)
 =t\mu\det(A_2+tB)-t^2m^{\mathsf T}\adj(A_2+tB)m
\]
has coefficient \(t\mu\Delta\).  Hence \(\mu=0\), and \(v=e_3\) satisfies
\(v^{\mathsf T}(A+x_1M)v=0\).

If the essential variable number is one, take
\(g_3=\kappa y_1^3\), \(\kappa\ne0\).  Let \(N\) be the lower-right
\(2\times2\) block of \(M\).  Then
\[
 \det\bigl(\operatorname{diag}(6\kappa y_1,0,0)+tM\bigr)
 =6\kappa y_1t^2\det N+t^3\det M.
\]
Thus \(\det N=0\).  Any nonzero
\(v\in\ker N\subset\operatorname{span}(e_2,e_3)\) satisfies
\(v^{\mathsf T}Av=v^{\mathsf T}Mv=0\).

Finally, if \(g_3=0\), then \(t^3\det M=0\), so any nonzero
\(v\in\ker M\) works.
\end{proof}

Apply the lemma to \(F=f_3\) and make a constant change among \(y\) so that
the fourth coordinate is \(v\).  Then \(\partial_{44}f_3=0\), while \(f_4\)
is independent of \(x_4\).  Expanding the quadratic and lower parts again
gives \eqref{eq:common-structure}.

\begin{proposition}[Degenerate structural reduction]\label{prop:degenerate-structure}
Both the genuinely binary and unary quartic cone tops reduce, after a constant
linear change, to \eqref{eq:common-structure}.
\end{proposition}

\section{The uniform inversion proposition}\label{sec:uniform}

\begin{proposition}[Uniform inversion]\label{prop:uniform}
Let
\[
 f=P(x_1,x_2,x_3)+uQ(x_1,x_2,x_3)+a u^2,
 \qquad \deg Q\leq2,
\]
where \(P\) is arbitrary.  If
\(\det\Hess(f)=c\in\CC^\times\), then \(\grad f\) is a polynomial
automorphism.
\end{proposition}

\begin{proof}
Write \(x=(x_1,x_2,x_3)^{\mathsf T}\) and set
\[
 M=\Hess_x(Q),\qquad
 \ell=\grad Q(0),\qquad
 v=\grad Q=Mx+\ell,\qquad
 B=\Hess_x(P)+uM.
\]
Then
\begin{equation}\label{eq:block-hessian}
 \Hess(f)=\begin{pmatrix}B&v\\v^{\mathsf T}&2a\end{pmatrix}.
\end{equation}

\subsection*{The branch \(a\ne0\)}

Put
\[
 z=u+\frac{Q(x)}{2a},\qquad
 R(x)=P(x)-\frac{Q(x)^2}{4a},\qquad
 S_z(x)=R(x)+zQ(x).
\]
The output coordinates satisfy
\[
 Y_4=Q(x)+2au=2az,\qquad
 (Y_1,Y_2,Y_3)=\grad_xS_z(x).
\]
The derivative in the original gradient is taken at fixed \(u\), whereas
\(\grad_xS_z\) is taken at fixed \(z\); the two expressions agree after
substitution.  Direct differentiation at fixed \(z\), followed by the
substitution \(z=u+Q/(2a)\), gives
\[
 \Hess_x(S_z)=B-\frac1{2a}vv^{\mathsf T}.
\]
The Schur complement of the lower-right entry in
\eqref{eq:block-hessian}, followed by the triangular substitution between
\((x,u)\) and \((x,z)\), yields the polynomial identity
\[
 \det\Hess_x(S_z)=\frac{c}{2a}\in\CC^\times.
\]

Over \(K=\CC(z)\), the known complex case \(\mathrm{HC}_3\), together with
Lemma \ref{lem:charzero-transfer}, gives a polynomial inverse of
\(\grad_xS_z\).
Its Jacobian determinant is already a unit over
\(R=\CC[z]\), so Lemma \ref{lem:descent} gives
\[
 (\grad_xS_z)^{-1}
 =H(z,Y_1,Y_2,Y_3)\in\CC[z,Y_1,Y_2,Y_3]^3.
\]
Consequently
\[
 z=\frac{Y_4}{2a},\qquad
 x=H(z,Y_1,Y_2,Y_3),\qquad
 u=z-\frac{Q(x)}{2a}
\]
is a polynomial inverse.

\subsection*{The branch \(a=0\)}

Now \(f=P+uQ\), and Lemma \ref{lem:block-det} gives
\begin{equation}\label{eq:uniform-block-det}
 c=-v^{\mathsf T}\adj\bigl(\Hess_x(P)+uM\bigr)v.
\end{equation}
The coefficient of \(u^2\) is
\[
 v^{\mathsf T}\adj(M)v=0.
\]
Lemma \ref{lem:adj-quadratic} and comparison of the degree-two part in \(x\)
imply that either \(M=0\) or \(\det M=0\).  Thus a nonzero \(M\) has rank one
or two.

\subsubsection*{Rank two}

After congruence, translation in \(\operatorname{im}M\), and removal of a
constant in \(Q\), we may assume
\[
 M=\operatorname{diag}(1,1,0),\qquad
 Q=\frac12(x_1^2+x_2^2),\qquad v=(x_1,x_2,0)^{\mathsf T}.
\]
Indeed, the constant part of
\(v^{\mathsf T}\adj(M)v=0\) first forces the kernel component of the affine
part of \(\grad Q\) to vanish.  Write \(t_{ij}=\partial_i\partial_jP\).
Expanding \eqref{eq:uniform-block-det} gives
\[
\det\Hess(f)= -\Bigl[
 x_1^2\bigl((t_{22}+u)t_{33}-t_{23}^2\bigr)
 +2x_1x_2(t_{13}t_{23}-t_{12}t_{33})
 +x_2^2\bigl((t_{11}+u)t_{33}-t_{13}^2\bigr)
 \Bigr].
\]
The coefficient of \(u\) is
\(-(x_1^2+x_2^2)t_{33}\), hence \(t_{33}=0\).  Therefore
\[
 P=R(x_1,x_2)+x_3S(x_1,x_2).
\]
The constant part becomes
\[
 c=(x_1S_{x_2}-x_2S_{x_1})^2,
\]
whose right-hand side vanishes at \(x_1=x_2=0\), a contradiction.

\subsubsection*{Rank one}

Put \(M=\operatorname{diag}(1,0,0)\) and rename the variables
\((x,y,z)\).  A translation in \(x\) removes the image component of the
affine part of \(\grad Q\); a linear change in the \((y,z)\)-plane makes its
kernel component \((\alpha,0)\), and a constant in \(Q\) only translates the
fourth output.  We may therefore normalize
\[
 Q=\frac12x^2+\alpha y,\qquad v=(x,\alpha,0)^{\mathsf T}.
\]
If \(\alpha=0\), then
\[
 c=-x^2(t_{22}t_{33}-t_{23}^2),
\]
which cannot be a nonzero constant.  Suppose \(\alpha\ne0\).  The coefficient
of \(u\) in \eqref{eq:uniform-block-det} is \(-\alpha^2t_{33}\), so
\[
 P=R(x,y)+zS(x,y).
\]
The constant part is
\begin{equation}\label{eq:rank-one-square}
 c=(xS_y-\alpha S_x)^2.
\end{equation}
Choose \(\kappa\in\CC^\times\) with \(\kappa^2=c\).  Since the polynomial ring
is a domain, change the sign of \(\kappa\) if necessary so that
\[
 xS_y-\alpha S_x=\kappa.
\]
Define
\[
 q=\frac12x^2+\alpha y,\qquad
 s=-\frac{x}{\alpha},\qquad
 D=x\partial_y-\alpha\partial_x.
\]
Then \(Dq=0\), \(Ds=1\), and \((q,s)\) is a polynomial coordinate system:
\[
 x=-\alpha s,\qquad
 y=\frac{q-\frac12x^2}{\alpha}.
\]
Thus \(D=\partial_s\) and
\[
 S=\kappa s+F(q),\qquad F\in\CC[t].
\]

The last two gradient outputs are \(Y_3=S\) and \(Y_4=q\).  Hence
\[
 q=Y_4,\qquad
 s=\frac{Y_3-F(Y_4)}{\kappa},\qquad
 x=-\alpha s,\qquad
 y=\frac{q-\frac12x^2}{\alpha}.
\]
Put
\[
 A_1=Y_1-R_x(x,y),\qquad A_2=Y_2-R_y(x,y).
\]
Then
\[
 \binom{A_1}{A_2}
 =\begin{pmatrix}S_x&x\\S_y&\alpha\end{pmatrix}
  \binom{z}{u},
 \qquad
 \det\begin{pmatrix}S_x&x\\S_y&\alpha\end{pmatrix}
 =\alpha S_x-xS_y=-\kappa.
\]
Thus \(z,u\) are recovered polynomially.  This solves the rank-one branch
rather than excluding it by a top-form hypothesis.

\subsubsection*{Rank zero}

Finally, \(M=0\) makes \(Q=\ell^{\mathsf T}x+q_0\) affine-linear.  If
\(\ell=0\), the last Hessian row and column vanish, so
\(\det\Hess(f)=0\), contrary to \(c\ne0\).  Otherwise take
\(\ell=(0,0,\lambda)\), \(\lambda\ne0\), by a linear change.  Equation
\eqref{eq:uniform-block-det} becomes
\[
 c=-\lambda^2\det\Hess_{x_1,x_2}(P)\in\CC^\times.
\]
Over \(K=\CC(x_3)\), the known complex case \(\mathrm{HC}_2\), together with
Lemma \ref{lem:charzero-transfer}, gives a polynomial inverse for
\[
 F=(P_{x_1},P_{x_2}).
\]
Lemma \ref{lem:descent}, with \(R=\CC[x_3]\), makes its coefficients
polynomial in \(x_3\).  The full inverse is then
\[
 x_3=\frac{Y_4-q_0}{\lambda},\qquad
 (x_1,x_2)=F^{-1}(x_3,Y_1,Y_2),\qquad
 u=\frac{Y_3-P_{x_3}(x_1,x_2,x_3)}{\lambda}.
\]
All ranks and both values of \(a\) are covered.
\end{proof}

\section{Proof of the main theorem}

\begin{proof}[Proof of Theorem \ref{thm:main}]
By Lemma \ref{lem:top-hessian},
\(\det\Hess(f_4)=0\).  Theorem \ref{thm:hesse} reduces \(f_4\) to at most
three variables.  The classification of Section \ref{sec:classification} is
exhaustive.

If the cone top is genuinely ternary, Proposition
\ref{prop:ternary-structure} gives the common structural form.  If it is
genuinely binary or unary, Proposition \ref{prop:degenerate-structure} gives
the same form.  Proposition \ref{prop:uniform} then proves that \(\grad f\) is
a polynomial automorphism in every case.
\end{proof}

\begin{corollary}[Quartic \(\mathrm{HC}_4\)]\label{cor:quartic-hc4}
The Hessian conjecture in dimension four holds for every potential of degree
at most four.
\end{corollary}

\begin{proof}
Degrees at most two are affine-gradient cases, degree three follows from
Wang's quadratic Keller-map theorem, and degree four is Theorem
\ref{thm:main}.
\end{proof}

\section{Discussion}

\subsection{What remains open}

Theorem \ref{thm:main} proves the complete quartic case, not the full
four-dimensional Hessian conjecture.  For degree at least five, the top form
is still a cone by Theorem \ref{thm:hesse}, but the determinant contains more
homogeneous layers and the degree-seven/six and degree-six/five arguments above
do not carry over verbatim.

Proposition \ref{prop:uniform}, however, has no degree bound on \(P\).
Consequently it already proves arbitrary-degree cases which admit a coordinate
form
\[
 P(x_1,x_2,x_3)+x_4Q(x_1,x_2,x_3)+a x_4^2,\qquad \deg Q\leq2.
\]

\subsection{Relation to the plane Jacobian conjecture}

For a Keller map \(F:\CC^n\to\CC^n\), the potential
\[
 \Phi(x,y)=y^{\mathsf T}F(x)
\]
satisfies
\[
 \det\Hess(\Phi)=(-1)^n(\det JF)^2,\qquad
 \grad\Phi(x,y)=\bigl(JF(x)^{\mathsf T}y,F(x)\bigr).
\]
Thus \(\mathrm{HC}_{2n}\) implies \(\mathrm{JC}_n\).  In particular, the full
\(\mathrm{HC}_4\) would imply the still-open \(\mathrm{JC}_2\), but the
quartic result here addresses only a bounded-degree part of
\(\mathrm{HC}_4\).

\section*{Acknowledgments}

Generative-AI systems---GPT-5.6 Sol, GPT-5.6 Luna, Claude Fable~5, and
DeepSeek V4 Pro---were used during the development of this work for exploring
candidate arguments, adversarial proof review, algebraic checking, and
editorial assistance.  The author independently reviewed all mathematical
arguments and references and assumes full responsibility for the contents of
the paper.  Lean~4 with Mathlib and SymPy were used for auxiliary verification
of selected algebraic identities; these computational checks are not part of
the logical proof.

\end{document}